\documentclass[12pt, reqno]{amsart}
\usepackage{amsmath, amsthm, amscd, amsfonts, amssymb, graphicx, color}
\usepackage[bookmarksnumbered, colorlinks, plainpages]{hyperref}
\hypersetup{colorlinks=true,linkcolor=red, anchorcolor=green, citecolor=cyan, urlcolor=red, filecolor=magenta, pdftoolbar=true}
\newtheorem{theorem}{Theorem}[section]
\newtheorem{lemma}[theorem]{Lemma}
\newtheorem{proposition}[theorem]{Proposition}
\newtheorem{corollary}[theorem]{Corollary}
\theoremstyle{definition}
\newtheorem{definition}[theorem]{Definition}

\theoremstyle{remark}
\newtheorem{remark}[theorem]{Remark}
\numberwithin{equation}{section}

\begin{document}

\title[On $\mathcal{AN}$-operators]{ On Absolutely Norm attaining operators}

\author[G. Ramesh]{G. Ramesh}

\address{Department of Mathematics\\I. I. T. Hyderabad, \\E-Block, 305, Kandi (V)\\ Sangareddy (M). Medak (Dist)\\  Telangana, India-502 285.}
\email{rameshg@iith.ac.in}

\author[D. Venku Naidu]{D. Venku Naidu}

\address{Department of Mathematics\\I. I. T. Hyderabad, \\E-Block, 307, Kandi (V)\\ Sangareddy (M). Medak (Dist)\\  Telangana, India-502 285.}
\email{venku@iith.ac.in}
\subjclass{Primary 47A75; Secondary 47A10}

\keywords{compact operator, norm attaining operator, $\mathcal{AN}$-operator, Fredholm operator, Fuglede theorem, Weyl's theorem}
\date{\today}

\begin{abstract}
We give necessary and sufficient conditions for a bounded operator defined between complex Hilbert spaces to be absolutely norm attaining. We discuss  structure of such operators in the case of self-adjoint and normal operators separately.
Finally, we discuss several properties of absolutely norm attaining operators.

\end{abstract}
\maketitle
\section{ Introduction and Preliminaries}
The class of absolutely norm attaining operators (shortly, $\mathcal{AN}$-operators) between complex Hilbert spaces were introduced  and discussed several important class of examples and properties of these operators  by Carvajal and Neves  in \cite{carvajalneves}. Later, a structure of these operators on separable Hilbert spaces is proposed in \cite{rameshstructurethm}. But, an  example of $\mathcal{AN}$-operator which does not fit into the characterization of \cite{rameshstructurethm} is given in \cite{SP} and the authors discussed the structure of positive $\mathcal{AN}$- operators between arbitrary Hilbert spaces. In this article, first, we give  necessary and sufficient conditions for an operator to be \textit{positive and $\mathcal{AN}$}. In fact, we show that  a bounded operator $T$ defined on an infinite dimensional Hilbert space is positive and $\mathcal{AN}$ if and only if there exists a unique triple $(K,F,\alpha)$, where $K$ is a positive compact operator, $F$ is a positive finite rank operator, $\alpha$ is positive real number such that  $T=K-F+\alpha I$ and $KF=0,\; F\leq \alpha I$ (See Theorem \ref{sharpcharacterization}). In fact, here $\alpha=m_e(T)$, the essential minimum modulus of $T$.  This is an improvement of \cite[Theorem 5.1]{SP}. Using this result,  we give explicit structure of self-adjoint  and $\mathcal{AN}$-operators as well as   normal and  $\mathcal{AN}$-operators. Finally, we also obtain structure of general $\mathcal{AN}$-operators. In the process we also prove several important properties of $\mathcal{AN}$-operators. All these results are new.

We organize the article as follows: In the remaining part of this section we explain the basic terminology,  notations and  necessary results that will be needed for proving main theorems. In the second section we give a  characterization of positive $\mathcal{ AN}$-operators and prove several important properties. In the third section we discuss the structure of self-adjoint and normal $\mathcal{AN}$-operators and in the fourth section we discuss about general $\mathcal{AN}$-operators.

Throughout the article we consider complex Hilbert spaces which will be denoted by $H,H_1,H_2$ etc.  The inner product and the
induced norm are denoted  by  $\langle, \rangle$ and $\|\cdot \|$ respectively. The unit sphere  of a closed subspace $M$ of $H$ is denoted by $S_M:={\{x\in M: \|x\|=1}\}$ and  $P_M$ denote the orthogonal projection $P_M:H\rightarrow H$ with range $M$. The identity operator on $M$ is denoted by $I_M$.

A linear operator $T: H_1\rightarrow H_2$ is said to be bounded if there exists a $k>0$ such that $\|Tx\|\leq k \|x\|$ for all $x\in H_1$.  If $T$ is bounded, the quantity $\|T\|=\sup{\{\|Tx\|: x\in S_{H_1}}\}$ is finite and is called the norm of $T$. We denote the space of all bounded linear operators between $H_1$ and $H_2$  by $\mathcal B(H_1,H_2)$. In case if $H_1=H_2=H$, then $\mathcal
B(H_1,H_2)$ is denoted by $\mathcal B(H)$.  For $T\in \mathcal B(H_1,H_2)$, there exists a unique operator denoted by $T^*:H_2\rightarrow H_1$ satisfying
 \begin{equation*}
 \langle Tx,y\rangle =\langle x, T^*y\rangle \; \text{for all}\; x\in H_1 \; \text{and} \; \text{for all}\; y\in H_2.
 \end{equation*}
 This operator $T^*$ is called the adjoint of $T$.  The null space and the range spaces of $T$ are denoted by $N(T)$ and $R(T)$ respectively.

Let $T\in \mathcal B(H)$. Then $T$ is said to be \textit{normal} if $T^*T=TT^*$, \textit{self-adjoint} if $T=T^*$. If $T$ is self-adjoint and $\langle Tx,x\rangle \geq 0$ for all $x\in H$, then $T$ is called \textit{positive}. It is well known that for a positive operator $T$, there exists a unique  positive operator $S\in \mathcal B(H)$ such that $S^2=T$. We write $S=T^{\frac{1}{2}}$  and is called as the \textit{positive square root} of $T$.

If $S,T\in \mathcal B(H)$ are self-adjoint and $S-T\geq 0$, then we write this by $S\geq T$.

If $P\in \mathcal B(H)$ is such that $P^2=P$, then $P$ is called a \textit{projection}. If $N(P)$ and $R(P)$ are orthogonal to each other, then $P$ is called an \textit{orthogonal projection}. It is a well known  fact  that a projection  $P$ is an orthogonal projection if and only if it is self-adjoint if and only if it is normal.

We call an operator $V\in \mathcal B(H_1,H_2)$ to be an \textit{isometry} if $\|Vx\|=\|x\|$ for each $x\in H_1$. An operator $V\in \mathcal B(H_1,H_2)$ is said to be a \textit{partial isometry} if $V|_{N(V)^{\bot}}$ is an isometry. That is $\|Vx\|=\|x\|$ for all $x\in N(V)^{\bot}$. If $V\in \mathcal B(H)$ is isometry and onto, then $V$ is said to be a \textit{unitary operator}.


In general, if $T\in \mathcal B(H_1,H_2)$, then $T^*T\in \mathcal B(H_1)$ is positive and $|T|:=(T^*T)^{\frac{1}{2}}$ is called the \textit{modulus} of $T$. In fact, there exists a unique partial isometry $V\in \mathcal B(H_1,H_2)$ such that
$T=V|T|$ and $N(V)=N(T)$. This factorization is called the \textit{polar decomposition} of $T$.

 If $T\in \mathcal B(H)$, then $T=\frac{T+T^*}{2}+i(\frac{T-T^*}{2i})$. The operators $Re(T):=\frac{T+T^*}{2}$ and $Im(T):=\frac{T-T^*}{2i}$ are self-adjoint and called the \textit{real} and the \textit{imaginary} parts of $T$ respectively.

A closed subspace $M$ of $H$ is said to be \textit{invariant} under $T\in \mathcal B(H)$ if $TM\subseteq M$ and \textit{reducing} if both $M$ and $M^\bot$ are invariant under $T$.

For $T\in \mathcal B(H)$, the set $$\rho(T):={\{\lambda \in \mathbb C: T-\lambda I:H\rightarrow H\; \text{ is invertible and}\;  (T-\lambda I)^{-1}\in \mathcal B(H) }\}$$ is called the resolvent set and the complement $\sigma(T)=\mathbb C\setminus \rho(T)$ is called the \textit{spectrum} of $T$. It is well known that $\sigma(T)$ is a non empty compact subset of $\mathbb C$. The point spectrum of $T$ is defined by
\begin{equation*}\sigma_p(T)={\{\lambda \in \mathbb C: T-\lambda I \; \text{is not one-to-one}}\}.
\end{equation*}
 Note that $\sigma_{p}(T)\subseteq \sigma(T)$.

A self-adjoint operator  $T\in \mathcal B(H)$  is positive if and only if  $\sigma(T)\subseteq [0,\infty)$.

If $T\in \mathcal B(H_1,H_2)$, then $T$ is said to be \textit{compact} if for every  bounded set $S$ of $H_1$, the set $T(S)$ is pre-compact in $H_2$. Equivalently, for every bounded sequence $(x_n)$ of $H_1$, $(Tx_n)$ has a convergent subsequence in $H_2$. We denote the set of all compact operators between $H_1$ and $H_2$ by $\mathcal K(H_1,H_2)$. In case if $H_1=H_2=H$, then $\mathcal K(H_1,H_2)$ is denoted by $\mathcal K(H)$.

A bounded linear operator $T:H_1\rightarrow H_2$ is called \textit{finite rank} if $R(T)$ is finite dimensional. The space of all finite rank operators between $H_1$ and $H_2$ is denoted by $\mathcal F(H_1,H_2)$ and we write $\mathcal F(H,H)=\mathcal F(H)$.

All the above mentioned basics of operator theory can be found in \cite{rudin,jbconway,halmosproblembook,taylorlay}.

An operator $T\in \mathcal B(H_1,H_2)$ is said to be \textit{norm attaining} if there exists a $x\in S_{H_1}$ such that $\|Tx\|=\|T\|$.  We denote the class of norm attaining operators by $\mathcal N(H_1,H_2)$. It is known that $\mathcal N(H_1,H_2)$ is dense in $\mathcal B(H_1,H_2)$ with respect to the operator norm of $\mathcal B(H_1,H_2)$. We refer \cite{enfloetal} for more details on this topic.

We say $T\in \mathcal B(H_1,H_2)$ to be  \textit{absolutely norm attaining}  or $\mathcal {AN}$-operator (shortly), if  $T|_M$,  the restriction of $T$ to $M $,  is norm attaining for every non zero closed subspace $M$ of $H_1$. That is $T|_M\in \mathcal N(M, H_2)$ for every non zero closed subspace $M$ of $H_1$ \cite{carvajalneves}. This class contains $\mathcal K(H_1,H_2)$, and  the class of partial isometries with finite dimensional null space or finite dimensional range space.

We have the following characterization of norm attaining operators:
\begin{proposition}\label{normattainingeigenvalue}\cite[Proposition 2.4]{carvajalneves}
Let $T\in \mathcal B(H)$ be self-adjoint. Then
\begin{enumerate}
\item $T\in \mathcal N(H)$ if and only if either $\|T\|\in \sigma_p(T)$ or $-\|T\|\in \sigma_p(T)$
\item if $T\geq 0$, then $T\in \mathcal N(H)$ if and only if $\|T\|\in \sigma_p(T)$.
\end{enumerate}
\end{proposition}

%

For $T\in \mathcal B(H_1,H_2)$, the quantity
\begin{equation*}
m(T):=\inf{\{\|Tx\|: x\in S_{H_1}}\}
\end{equation*}
is called the minimum modulus of $T$.
If $H_1=H_2=H$ and $T^{-1}\in \mathcal B(H)$, then $m(T)=\dfrac{1}{\|T^{-1}\|}$ (see \cite[Theorem 1]{bouldin} for details).

The following definition is available in \cite{schmudgen} for densely defined closed operators (not necessarily bounded) on a Hilbert space, and this holds true automatically for bounded operators.
\begin{definition}\cite[Definition 8.3 page 178]{schmudgen}
Let $T=T^*\in \mathcal B(H)$. Then the \textit{discrete spectrum} $\sigma_d(T)$ of $T$ is defined as the set of all eigenvalues of $T$ with finite multiplicities which are isolated points of the spectrum $\sigma(T)$ of $T$. The complement set $\sigma_{ess}(T)=\sigma(T)\setminus \sigma_d(T)$ is called the \textit{essential spectrum} of $T$.
\end{definition}

By the Weyl's theorem we can assert that if $T=T^*$ and $K=K^*\in \mathcal K(H)$, then $\sigma_{ess}(T+K)=\sigma_{ess}(T)$ (see \cite[Corollary 8.16, page 182]{schmudgen} for details). If $H$ is a separable Hilbert space, the \textit{essential minimum modulus} of $T$ is defined to be $m_e(T):=\inf{\{\lambda: \lambda \in \sigma_{ess}(|T|)}\}$ (see  \cite{bouldin} for details). The same result in the general case is dealt in \cite[Proposition 2.1]{Feshchenko}.

Let $H=H_1\oplus H_2$ and $T\in \mathcal B(H)$. Let $P_j:H\rightarrow H$ be an  orthogonal projection onto $H_j$ for $j=1,2$. Then
                                                    $ T=\left(
                                                       \begin{array}{cc}
                                                         T_{11} &T_{12} \\
                                                         T_{21} &T_{22} \\
                                                       \end{array}
                                                     \right)$,
                                                     where $T_{ij}:H_j\rightarrow H_i$ is the operator given by $T_{ij}=P_iTP_j|_{H_j}$.  In particular, $T(H_1)\subseteq H_1$ if and only if $T_{12}=0$. Also, $H_1$ reduces $T$ if and only if
                                                     $T_{12}=0=T_{21}$ (for details see \cite{taylorlay,jbconway}.

\section{Positive $\mathcal{AN}$-operators}

In this section we describe the structure of  operators which are  positive and satisfy the $\mathcal{AN}$-property. First, we recall results which are  necessary for proving our results.

\begin{theorem}\cite[Theorem 5.1]{SP}\label{structurethm2nd}
Let $H$ be a complex Hilbert space of arbitrary dimension and let $P$ be a positive operator
on $H.$ Then $P$ is an $\mathcal{AN}$- operator iff $P$ is of the form $P = \alpha I + K + F$,
where $\alpha\geq 0$, $K$ is a positive compact operator and $F$ is self-adjoint finite
rank operator.
\end{theorem}

 \begin{theorem}\cite[Theorem 3.8]{SP}\label{SPthm} Let $T\in B(H)$ be positive and $T\in \mathcal {AN}(H)$. Then
\begin{equation}
T=\displaystyle \sum_{\alpha \in \Lambda} \beta_{\alpha}v_{\alpha}\otimes v_{\alpha},
\end{equation}
where ${\{v_{\alpha}:\alpha\in \Lambda}\}$ is an orthonormal basis consisting of entirely eigenvectors of $T$ and for every $\alpha\in \Lambda$, $Tv_{\alpha}=\beta_{\alpha}v_{\alpha}$ with $\beta_{\alpha}\geq 0$ such that
\begin{enumerate}
\item for every non empty set $\Gamma$ of $\Lambda$, we have
\begin{equation*}
\sup{\{\beta_{\alpha}:\alpha\in \Gamma}\}=\max{\{\beta_{\alpha}:\alpha \in \Gamma}\}
\end{equation*}
\item\label{uniqueltpt-SP} the spectrum $\sigma(T)=\overline{{\{\beta_{\alpha}:\alpha \in \Lambda}\}}$ has at most one limit point. Moreover, this unique limit point (if exists) can only be the limit of an increasing sequence in the spectrum
\item \label{uniqueeigenvalueinfinitemult}the set ${\{\beta_{\alpha}:\alpha \in \Lambda}\}$ of eigenvalues of $T$, without counting multiplicities, is countable and has atmost one eigenvalue with infinite multiplicity
\item \label{ltptinfiniteeigevalue} if  $\sigma(T)$ has  both, a limit point and an eigenvalue with infinite multiplicity, then they must be same.
\end{enumerate}
(Here $(v_{\alpha}\otimes v_{\alpha})(x)=\langle x,v_{\alpha}\rangle v_{\alpha}$ for each $\alpha \in \Lambda$ and for each  $x\in H$).
\end{theorem}



\begin{lemma}\label{nullspacecomparison}
Let $S,T\in \mathcal B(H)$ be positive such that $S\leq T$. Then $N(T)\subseteq N(S)$.
\end{lemma}
\begin{proof}
If $x\in H$, then $\|S^{\frac{1}{2}}x\|^2=\langle Sx,x\rangle\leq \langle Tx,x\rangle=\|T^{\frac{1}{2}}x\|^2$. By observing the fact that for any $A\geq0$, $N(A^{\frac{1}{2}})=N(A)$, the conclusion follows.
\end{proof}

\begin{theorem}\label{sharpcharacterization}
Let $H$ be an infinite dimensional Hilbert space and $T\in \mathcal B(H)$. Then the following statements are equivalent:

 \begin{enumerate}
   \item \label{AN-property} $T\in \mathcal {AN}(H)$ and positive
   \item \label{orthogonalrepn}  there exists a  unique triple $(K,F,\alpha )$ where
   \begin{itemize}
     \item[(a)] $K\in \mathcal K(H)$ is positive
     \item [(b)]$F\in \mathcal F(H)$  and $0\leq F\leq \alpha I$
     \item [(c)] $KF=0$
   \end{itemize}
   such that  $T=K-F+\alpha I$.
   \end{enumerate}
\end{theorem}
\begin{proof}
Proof of $(\ref{AN-property})\Rightarrow (\ref{orthogonalrepn})$: By Theorem \ref{structurethm2nd}, $T=K'-F'+\alpha I$, where $K'\in \mathcal K(H)$ is positive, $F'=F'^{*}\in \mathcal F(H)$  and $\alpha \geq 0$.
     Next we claim that $K'F'=0$. This readily follows by the  proof in \cite[Theorem 5.1]{SP}.

     Now, let $F'=F'_{+}-F'_{-}$ be the decomposition of $F'$ in terms of positive operators  $F'^{+}$ and $F'_{-}$, respectively (see \cite[page 180]{groetsch} for details). Note that $F'_{+}F'_{-}=0$. Write $K=K'+F'_{-}$ and $F=F'_{+}$. Then $K\geq 0$ and $F\geq 0$. Since $K'F'=0$,  it follows that $K'|F'|=0$. That is $K'(F'_{+}+F'_{-})=0$. Also, $K'(F'_{+}-F'_{-})=0$. These two equations imply that $KF=0$. As $T\geq 0$ and $F\geq 0$ such that $TF=FT$, it follows that $FT\geq 0$.  But $FT=F(\alpha I-F)$. Let $\lambda\in \sigma(F)$. Then $\lambda\geq 0$ and since $FT\geq 0$,  by the spectral mapping theorem, we have that $\lambda(\alpha-\lambda)\geq 0$. From this, we can conclude that $\alpha-\lambda\geq 0$ for each  $\lambda \in \sigma(F)$. As $\alpha I-F$ is self-adjoint and $\sigma(\alpha I-F)\subseteq [0,\infty)$, $\alpha I-F$ must be positive.
This concludes that $F\leq \alpha I$.

Next we  show that the triple satisfying the given conditions is unique. Suppose there exists two triples $(K_1,F_1,\alpha_1), (K_2,F_2,\alpha_2 )$ satisfying the stated conditions. We prove this by considering all possible cases.

Case $1;\; \alpha_1=0:$ In this case, $F_1=0$. Hence $K_1=T=K_2-F_2+\alpha_2I$. This shows that $\alpha_2I=K_1-K_2+F_2$, a compact operator. Since $H$ is infinite dimensional, it follows that $\alpha_2=0$. Thus $F_2=0$. Hence we can conclude that $K_1=K_2$.

Case $2;\; F_1=0, \alpha_1>0$: In this case,
\begin{equation}\label{uniquenesseq1}
K_1+\alpha_1I=K_2-F_2+\alpha_2I.
 \end{equation}
 Then $(\alpha_2-\alpha_1)I=(K_1-K_2)+F_2$ ,  a compact operator. If this is zero, then $\alpha_1=\alpha_2$. If not, $(\alpha_1-\alpha_2)I$ is a compact operator and $H$ is infinite dimensional, $\alpha_1=\alpha_2$.

Now, the Equation (\ref{uniquenesseq1}) can be written as $K_2=F_2+K_1\geq F_2$. Now, by   Lemma  \ref{nullspacecomparison}, we have that $N(K_2)\subseteq N(F_2)$. But, by the condition $K_2F_2=0$, we have, $R(F_2)\subseteq N(K_2)$, hence $R(F_2)\subseteq N(F_2)$. Thus, $F_2=0$.
From this we can conclude that $K_1=K_2$.

Case $3\; K_1=0, \; F_1\neq 0,\; \alpha_1>0$: We have $F_1+\alpha_1 I=K_2-F_2+\alpha_2 I$. Using the same argument as in the above cases, we can conclude that $\alpha_1=\alpha_2$. Thus we have $F_2=K_2+F_1\geq K_2$. Now, by Lemma  \ref{nullspacecomparison}, $N(F_2)\subseteq N(K_2)$. But by  the property $K_2F_2=0$, it follows that $R(F_2)\subseteq N(K_2)$. Hence $H=N(F_2)\oplus R(F_2)\subseteq N(K_2)$. This shows that $K_2=0$. Finally, using this we can get $F_1=F_2$.

Case $4 \; K_1\neq 0,\; F_1\neq 0,\; \alpha_1>0$: We can prove $\alpha_1=\alpha_2$ by arguing as in the earlier cases. With this we have
\begin{equation}\label{uniquenesseq2}
K_1-F_1=K_2-F_2.
 \end{equation}
 As $F_1$ commute with $K_1$ and $F_1$, it commute with $K_2-F_2$. So $F_1$ must commute with $(K_2-F_2)^2=K_2^2+F_2^2=(K_2+F_2)^2$. Thus, it commute with $K_2+F_2$. Hence we can conclude that $F_1$ commute with both $K_2$ and $F_2$. Since $N(F_1)$ is invariant under $K_1$ and $F_1$, by Equation (\ref{structurethm2nd}), $N(F_1)$ is invariant under $K_2-F_2$.

 Now if $x\in N(F_1)$. Then by Equation (\ref{uniquenesseq2}), we have $(K_2-K_1)x=F_2x$. Using the fact that $F_2\geq 0$, we can conclude that $K_2\geq K_1$ on $N(F_1)$. We also show that this will happen on $R(F_1)$.

 For $x\in H$, we have $F_1x\in R(F_1)$. Now,
 \begin{equation*}
\langle (F_2-F_1)(F_1x),F_1x\rangle =\langle (K_2-K_1)(F_1x),F_1x\rangle=\langle K_2(F_1x),F_1x\rangle\geq 0.
 \end{equation*}
 This shows that $K_2-K_1=F_2-F_1\geq 0$ on $R(F_1)$. Combining with the earlier argument, we can conclude that $K_1\leq K_2$. Now, interchanging the roles of $K_1$ and $K_2$, we can conclude that  $K_2\leq K_1$  and hence $K_1=K_2$. By Equation (\ref{uniquenesseq2}), we can conclude that $F_1=F_2$.

Proof of $(\ref{orthogonalrepn})\Rightarrow (\ref{AN-property})$: If $T=K-F+\alpha I$, where $K\in \mathcal K(H)$ is positive, $F\in \mathcal F(H)$ is positive, $\alpha \geq 0$ and $KF=0$. Then by Theorem \ref{structurethm2nd}, $T\in \mathcal {AN}(H)$. Since $K\geq 0$ and $-F+\alpha I\geq 0$, $T$ must be positive.
\end{proof}

\begin{remark}
Let $T$ be as in Theorem \ref{sharpcharacterization}. Then we have the following:
\begin{enumerate}
\item if $\alpha=0$, then $F=0$ and hence $T=K$. In this case $\sigma_{ess}(T)={\{\alpha}\}$
\item if $\alpha>0$ and $F=0$, then $T=K+\alpha I$. In this case, $\sigma_{ess}(T)={\{\alpha}\}$ and $m_{e}(T)=\alpha=m(T)$
\item  if $\alpha>0,\; K=0$ and $F\neq 0$, then $T=\alpha I-F$. In this case also, $\sigma_{ess}(T)={\{\alpha}\}$ and $m_{e}(T)=\alpha$
\item if $\alpha>0,\; F\neq 0$ and $K\neq 0$, then by the Weyl's theorem, $\sigma_{ess}(T)={\{\alpha}\}$ and $m_{e}(T)=\alpha$
\item if $\alpha=0$ and $K=0$, then $T=0$
\item if $N(T)$ is infinite dimensional, then $0$ is an eigenvalue with infinite multiplicity and hence $\alpha=0$ by  Theorem \ref{SPthm}. In this case, $F=0$ and hence $T=K$.
\end{enumerate}
\end{remark}

\begin{remark}
If  we take $F=0$ in Theorem \ref{sharpcharacterization}, then we get the structure obtained in \cite{rameshstructurethm}.
\end{remark}
Here we prove some important properties of $\mathcal{AN}$-operators.
\begin{proposition}\label{propertiesofAN}
 Let $T=K-F+\alpha I$, where $K\in \mathcal K(H)$ is positive, $F\in \mathcal F(H)$ is positive with $KF=0$ and $F\leq \alpha I$. Assume that $\alpha>0$. Then
 the following statements hold.
 \begin{enumerate}
  \item\label{closedrange} $R(T)$ is closed
  \item\label{fdmlnullspace}  $N(T)$ is finite dimensional
  \item \label{nullspaceinclusion} $N(T)\subseteq N(K)$
  \item\label{eigenvaluefiniterank} $Fx=\alpha x$ for all $x\in N(T)$. Hence $N(T)\subseteq R(F)$.  In this case, $\|F\|=\alpha$.
  \item  \label{injectivecharacterization} $T$ is one-to-one if and only if $\|F\|<\alpha$
  \item  \label{FredolmAN}$T$ is Fredholm and $m_e(T)=\alpha$.
  \end{enumerate}
\end{proposition}
\begin{proof}
Proof of (\ref{closedrange}): Since $K-F$ is a compact operator, $R(T)$ is closed. Here we have used the fact that for any $A\in \mathcal K(H)$, and $\lambda \in \mathbb C\setminus {\{0}\}$, $R(K+\lambda I)$ is closed.

Proof of (\ref{fdmlnullspace}): Let $x\in N(T)$. Then
\begin{equation}\label{differencecompacteqn}
(K-F)x=-\alpha x.
\end{equation}
That is $\alpha I_{N(T)}$ is compact. This concludes that $N(T)$ is finite dimensional.

Proof of (\ref{nullspaceinclusion}):  Let $x\in N(T)$. Multiplying Equation (\ref{differencecompacteqn}) by $K$ and using the fact that $KF=FK=0$, we have $K^2x=-\alpha Kx$. If $Kx\neq 0$, then $-\alpha \in \sigma_p(K)$, contradicts the positivity of $K$. Hence $Kx=0$.

Proof of (\ref{eigenvaluefiniterank}): Clearly, if $Tx=0$, then by (\ref{nullspaceinclusion}), we have $Fx=\alpha x$. This also concludes that $N(T)\subseteq R(F)$.

Proof of (\ref{injectivecharacterization}): If $T$ is not one-to-one, then $Fx=\alpha x$ for $x\in N(T)$ by (\ref{eigenvaluefiniterank}). Suppose  $T$ is one-to-one and $\|F\|=\alpha$. Since $F$ is norm attaining by Proposition \ref{normattainingeigenvalue}, there exists $x\in S_H$
such that $Fx=\alpha x$. Then $Tx=Kx-Fx+\alpha x=Kx$. But $KF=0$ implies that $x\in N(K)$.  So,  $Tx=Kx$=0.  By the injectivity of $T$, we have that  $x=0$. This contradicts the fact that $x\in S_H$. Hence  $\|F\|<\alpha$.

Proof of (\ref{FredolmAN}):  Note that $\sigma_{ess}(T)={\{\alpha}\}$ by the  Weyl's theorem on essential spectrum. Hence $m_e(T)=\alpha=m_e(T^*)$. Now $T$ is Fredholm operator  by  \cite[Theorem 2]{bouldin} with index zero.
\end{proof}

\begin{theorem}\label{squarerootpreserving}
Let $T\in \mathcal B(H)$ and positive. Then $T\in \mathcal {AN}(H)$ if and only if $T^2\in \mathcal {AN}(H)$.
\end{theorem}
\begin{proof}
First we will assume that $T\in \mathcal {AN}(H).$ Then  there exists a triple $(K, F,\alpha)$  as in (\ref{orthogonalrepn}) of Theorem \ref{sharpcharacterization}.  Then $T^2=K_1-F_1+\beta I$, where $K_1=K^2+2\alpha K$, a positive compact operator, $F_1=2\alpha F-F^2=(2\alpha I-F)F$ and $\beta=\alpha^2$. Clearly, $F_1\geq 0$ as it is the
product of two commuting positive operators. Also $F_1\in \mathcal F(H)$. Next, we show that $F_1\leq \alpha ^2I$. Clearly, $\alpha^2I-F_1$ is self-adjoint and $\alpha^2I-F_1=(\alpha I-F)^2\geq 0$. It can be easily verified that $K_1F_1=0$. So, $T^2$ is also in the same form. Hence by Theorem \ref{sharpcharacterization},  $T^2\in \mathcal {AN}(H)$.

Now, let $T^2\in \mathcal {AN}(H)$. Then by Theorem \ref{sharpcharacterization}, $T^2= K-F + \alpha I$, where $K\in \mathcal K(H)$ is positive, $F\in \mathcal F(H)$ is positive with $FK=KF=0$ and $F\leq \alpha I $.
If $\alpha > 0,$ then $(T-\sqrt{\alpha} I)(T+\sqrt{\alpha} I)= K-F$.  Since $T$ is positive  $T+\sqrt{\alpha} I$ is a positive invertible operator. Hence $T-\sqrt{\alpha} I= (K-F)(T+\sqrt{\alpha} I)^{-1}$.
Hence there is a positive compact operator, namely $K_1 = K (T+\sqrt{\alpha})^{-1}$ and a finite rank positive operator, namely $F_1 = F (T+\sqrt{\alpha} I)^{-1},$  such that $T-\sqrt{\alpha} I= K_1 + F_1.$
Hence $T= K_1-F_1 + \sqrt{\alpha}I$. Also note that since $F$ and $K$ commute with $T^2$, hence with $T$. Thus, we can conclude that $F_1K_1=0$.
Finally,
\begin{align*}
\|F_1\|\leq \|F\|\, \|(T+\sqrt{\alpha} I)^{-1}\| &\leq \alpha \, \frac{1}{m(T+\sqrt{\alpha} I)}\\
&=\frac{\alpha}{\sqrt{\alpha}+m(T)} \\
&\leq \frac{\alpha}{\sqrt{\alpha}}=\sqrt{\alpha}.
\end{align*}
In the third step of the above inequalities we used the fact that $m(T+\sqrt{\alpha} I)=\sqrt{\alpha}+m(T)$, which follows by \cite[Proposition 2.1]{rameshstructurethm}.

If $\alpha=0$, then clearly $F=0$ and hence $T^2=K$. So,  $T=K^{\frac{1}{2}}$, a compact operator which is clearly an $\mathcal{AN}$-operator.
\end{proof}
\begin{corollary}
Let $T\in \mathcal B(H)$ and positive. Then $T\in \mathcal {AN}(H)$ if and only if $T^{\frac{1}{2}}\in \mathcal {AN}(H)$.
\end{corollary}
\begin{proof}
Let $S=T^{\frac{1}{2}}$. Then $S\geq 0$. The conclusion follows by Theorem \ref{squarerootpreserving}.
\end{proof}
\begin{corollary}\label{gramchar}
Let $T\in \mathcal B(H_1,H_2)$. Then $T\in \mathcal {AN}(H_1,H_2)$ if and only $T^*T\in \mathcal {AN}(H_1)$.
\end{corollary}
\begin{proof}
Proof follows from the following:
$T^*T\in \mathcal {AN}(H_1)\Leftrightarrow |T|^2\in \mathcal {AN}(H_1)\Leftrightarrow |T|\in \mathcal {AN}(H_1) \Leftrightarrow T \in \mathcal {AN}(H_1, H_2).$
\end{proof}
We have the following consequence.
\begin{theorem}
 Let $T\in \mathcal {AN}(H)$ be self-adjoint and $\lambda $ be a purely imaginary number. Then $T \pm \lambda I \in \mathcal{AN}(H)$.
\end{theorem}
\begin{proof}
Let $S=T  \pm\lambda I$. Then $S^*S=T^2+|\lambda|^2I =K-F+(\alpha +|\lambda|^2)I$, where the triple $(K,F,\alpha)$ satisfy conditions  (\ref{orthogonalrepn}) of Theorem \ref{sharpcharacterization}.  Hence by Corollary \ref{gramchar}, $S\in \mathcal{AN}(H)$.
\end{proof}

The following result is well known.
\begin{lemma}\label{inversedifference}
  Let $S,T\in \mathcal B(H)$ be such that $S^{-1},T^{-1}\in \mathcal B(H)$. Then $S^{-1}-T^{-1}=T^{-1}(T-S)S^{-1}$.
\end{lemma}

\begin{theorem}
  Let $T=K-F+\alpha I$, where $(K, F, \alpha)$ satisfy conditions (\ref{orthogonalrepn}) of Theorem \ref{sharpcharacterization}. Then
  \begin{enumerate}
    \item \label{reducingfiniterank} $R(F)$ reduces $T$
    \item\label{matrixform} $T=\left(
               \begin{array}{cc}
                 K_0+\alpha I |_{N(F)}& 0 \\
                  & \alpha I|_{R(F)}-F_0 \\
               \end{array}
             \right)
    $,
    where $K_0=K|_{N(F)}$ and $F_0=F|_{R(F)}$.
    \item \label{blockforminverse} if $T$ is one-to-one and $\alpha>0$,  then $T^{-1}\in \mathcal B(H)$ and
    \tiny
    \begin{equation*}
      T^{-1}=\left(
                                               \begin{array}{cc}
                                                 \alpha ^{-1}I_{N(F)}-\alpha^{-1}K_0(K_0+\alpha I_{N(F)})^{-1} & 0 \\
                                                 0 & \alpha^{-1}I_{R(F)}+\alpha^{-1}F_0(\alpha I_{R(F)}-F_0)^{-1} \\
                                               \end{array}
                                             \right).
                                           \end{equation*}
  \end{enumerate}
\end{theorem}
\begin{proof}
Proof of (\ref{reducingfiniterank}):   First note that $T\geq 0$ and $T\in \mathcal{AN}(H)$. Let $y=Fx$ for some $x\in H$. Then $Ty=TFx=(K-F+\alpha I)Fx=(\alpha I-F)(Fx)=F(\alpha I-F)x\in R(F)$. This shows that $R(F)$ is invariant under $T$. As $T$ is positive, it follows that $R(F)$ is a reducing subspace for $T$.

  Proof of (\ref{matrixform}):  First,  we show that $K_0$ is a map on $N(F)$. For this we show that $N(F)$ invariant under $K$.  If $x\in N(F)$, then
  $FKx=0$ since $FK=0$. This proves that $N(F)$ is invariant under $K$. Thus $K_0\in \mathcal K(N(F))$. Also, clearly, $R(F)$ is invariant under $F$. Thus $F_0: R(F)\rightarrow R(F)$ is a finite dimensional operator. With respect to the pair of subspaces $(N(F),R(F))$,  $K$ has the decomposition:
  $$\left(
     \begin{array}{cc}
       K_0 & 0 \\
       0 & 0 \\
     \end{array}
   \right)
  .$$
  Similarly the operators $F$ and $\alpha I$ has the following block matrix forms respectively:
  $$\left(
     \begin{array}{cc}
       0 & 0 \\
       0 & F_0 \\
     \end{array}
   \right) \; \; \text{and}\;\; \left(
     \begin{array}{cc}
       \alpha I_{N(F)} & 0 \\
       0 & \alpha I_{R(F)} \\
     \end{array}
   \right).$$

With these representation of $K, F$ and $\alpha I$, by definition, $T$  can be represented as in  (\ref{matrixform}).

Proof of (\ref{blockforminverse}): By (\ref{closedrange}) of Proposition \ref{propertiesofAN}, $R(T)$ is closed. As $T$ is one-to-one, $T$  is bounded below. Since $T$ is positive, $T^{-1}\in \mathcal B(H)$.
 In this case $\|F_0\|=\|F\|<\alpha$, by  (\ref{injectivecharacterization}) of Proposition  \ref{propertiesofAN}. Hence we have
 \begin{equation}\label{blockforminverseeq1}
 T^{-1}=\left(
             \begin{array}{cc}
               (K_0+\alpha I_{N(F)})^{-1} & 0 \\
               0 & (\alpha I_{R(F)}-F_0)^{-1} \\
             \end{array}
           \right).
           \end{equation}

  By Lemma \ref{inversedifference},  we have
  \begin{equation*}
  (K_0+\alpha I_{N(F)})^{-1}-\alpha^{-1} I_{N(F)}=\alpha^{-1}I_{N(F)}-\alpha^{-1}K_0(K_0+\alpha I_{N(F)})^{-1},
  \end{equation*}
  and hence
  \begin{equation*}
  (K_0+\alpha I_{N(F)})^{-1}=\alpha^{-1} I_{N(F)}-\alpha^{-1}I_{N(F)}-\alpha^{-1}K_0(K_0+\alpha I_{N(F)})^{-1}.
    \end{equation*}
    Substituting these quantities in Equation \ref{blockforminverseeq1}, we obtain the representation of $T^{-1}$ as in (\ref{blockforminverse}).
  \end{proof}

  \begin{remark}
  Let
  \begin{align*}
  \beta & =\alpha^{-1},\\
   K_1&= \left(
                                     \begin{array}{cc}
                                       \alpha^{-1}K_0(K_0+\alpha I_{N(F)})^{-1} & 0 \\
                                       0 & 0 \\
                                     \end{array}
                                   \right)\\ \intertext{and}    F_1&=\left(
              \begin{array}{cc}
                 0& 0 \\
                0 & \alpha^{-1}F_0(\alpha I_{R(F)}-F_0)^{-1} \\
              \end{array}
            \right).
            \end{align*}

             Then $T^{-1}=\beta I-K_1+F_1$. Note that $\|K_1\|\leq \beta$, since $\|K_0(\alpha I_{N(F)}+K_0)^{-1}\|\leq 1$. Clearly, by definition, $K_1F_1=0$. This is exactly, the structure of absolutely minimum attaining  operators (shortly $\mathcal {AM}$-operators) in case when $T$ is positive and one-to-one. We refer \cite{GRS2} for more details of the structure of these operators. We recall that  $A\in \mathcal B(H_1,H_2)$ is said to be minimum attaining if there exists $x_0\in S_{H_1}$ such that $\|Ax_0\|=m(A)$ and absolutely minimum attaining if $A|_{M}$ is minimum attaining for each non zero closed subspace $M$ of $H_1$.
  \end{remark}

\begin{proposition}
Let $T\in \mathcal B(H)$ be satisfying conditions in Theorem  \ref{sharpcharacterization}. Then with respect the  pair of subspace $(N(K), N(K)^{\bot})$, $T$ has the following decomposition:
$$T=\left(
      \begin{array}{cc}
        \alpha I_{N(K)}-F_0 & 0 \\
        0 & K_0+\alpha I_{N(K)^{\bot}} \\
      \end{array}
    \right),
$$
where $F_0=F|_{N(K)}$ and $K_0=K|_{N(K)^{\bot}}$.
\end{proposition}
\begin{proof}
First we show that $N(K)$ is a reducing subspace for $T$. We know by Theorem \ref{sharpcharacterization}, that $T$ is positive. Hence  it suffices to show that $N(K)$ is invariant under $T$. For this, let $x\in N(K)$. Then $Tx=(\alpha I-F)(x)$
and $K(Tx)=(\alpha I-F)(Kx)=0$. This proves the claim. Next, if $x\in N(K)$, then $Tx=(\alpha I-F)(x)$. That is $T|_{N(K)}=I_{N(K)}-F|_{N(K)}$.

If $y\in N(K)^{\bot}=\overline{R(K)}$, then there exists a sequence $(x_n)\subset H$ such that $y=\displaystyle \lim_{n\rightarrow \infty}Kx_n$. So $Fy=\displaystyle \lim_{n\rightarrow  \infty}FKx_n=0$. Thus we have $Ty=Ky+\alpha y$. So $T|_{N(K)^{\bot}}=K_{N(K)^{\bot}}+\alpha I_{N(K)^{\bot}}$.
\end{proof}

\section{Self-adjoint and Normal $\mathcal{AN}$-operators}

In this section,  first we discuss the structure of self-adjoint $\mathcal{AN}$-operators. Later, we extend this to the case of normal operators.

\begin{theorem}\label{diagonalizabilityselfadj}
Let $T=T^*\in \mathcal {AN}(H)$. Then there exists an orthonormal basis consisting of eigenvectors of $T$.
\end{theorem}
\begin{proof}
The proof follows in the similar lines of \cite[Theorem 3.1]{SP}. For the sake of completeness we provide the details here.
Let $\mathcal B=\{x_\alpha: \alpha\in I\}$ be the maximal set of orthonormal eigenvectors of $T.$ This set is non empty, as $T=T^*\in \mathcal{AN}(H)$. Let $M=\overline{\mbox{span}}\{x_\alpha: \alpha \in I\}$. Then we claim  that $M=H$.
If not, $M^{\bot}$ is a proper non-zero closed subspace of $H$ and it is invariant under $T$.  Since $T=T^* \in \mathcal{AN}(H)$,  then we have either $||T|M^{\bot}||$ or $-||T|M^{\bot}||$ is an eigenvalue for $T|M^{\bot}$. Hence there is a
non-zero vector, say $x_0$ in $M^{\bot}$, such that $Tx_0=\pm ||T|M^{\bot}|| x_0.$ Since $M\cap M^{\bot}={\{0}\},$ we have arrived to a contradiction to the  maximality of $\mathcal B$.
\end{proof}
\begin{proposition}\label{eigenvaluepossibilities}
Let $T=T^*\in \mathcal {AN}(H)$. Then the following holds:
\begin{enumerate}
\item \label{infinitemult} $T$ can have atmost two eigenvalues with infinite multiplicity. Moreover, if $\alpha$ and $\beta$ are such eigenvalues, then $\alpha=\pm\beta$
\item \label{ltptinfinitemult}  if $T$ has an eigenvalue $\alpha$ with infinite multiplicity and $\beta$ is a limit point of $\sigma(T)$, then $\alpha=\pm\beta$
\item  \label{ltptunique}$\sigma(T)$ can have atmost two limit points. If $\alpha$ and $\beta$ are such points, then $\alpha=\pm\beta$.
\end{enumerate}
\end{proposition}
\begin{proof}
Proof of (\ref{infinitemult}): Let $\alpha_j\in \sigma_p(T)$ be such that $N(T-\alpha_jI)$ is infinite dimensional for each $j=1,2,3$.  Then $\alpha_j^2\in N(T^2)$ and  we have
$N(T-\alpha_jI)\subseteq N(T^2-\alpha_j^2I)$ for each $j=1,2,3$. Since $T^2\in \mathcal{AN}(H)$ and positive,  by (\ref{uniqueeigenvalueinfinitemult}) of Theorem \ref{SPthm}, it follows that $\alpha_1^2=\alpha_2^2=\alpha_3^2$. Thus $\alpha_1=\pm \alpha_2=\pm \alpha_3$.

Proof of (\ref{ltptinfinitemult}): Let $\alpha \in \sigma_p(T)$ with infinite multiplicity and $\beta \in \sigma(T)$, which is a limit point. Since $\sigma(T^2)={\{\lambda^2:\lambda \in \sigma(T)}\}$, it follows that $\alpha^2$ is an eigenvalue of
$T^2$ with infinite multiplicity as $N(T-\alpha I)\subseteq N(T^2-\alpha^2I)$ and $\beta^2$ is a limit point $\sigma(T^2)$. Since $T^2\in \mathcal{AN}(H)$ is positive, by (\ref{ltptinfiniteeigevalue}) of Theorem  (\ref{SPthm}), $\alpha^2=\beta^2$. Thus $\alpha=\pm \beta$.

Proof of (\ref{ltptunique}): Let $\alpha,\beta \in \sigma(T)$ be limit points of $\sigma(T)$. Then $\alpha^2,\beta^2\in \sigma(T^2)$ are limit points of $\sigma(T^2)$ and since $T^2\in \mathcal{AN}(H)$ and positive, by  (\ref{uniqueltpt-SP}) of Theorem \ref{SPthm}, $\alpha^2=\beta^2$, concluding $\alpha=\pm \beta$. By arguing as in Proof of (\ref{infinitemult}), we can show that there are at most two limit points for the spectrum.
\end{proof}




Let $T=T^*\in \mathcal B(H)$ and have the polar decomposition $T=V|T|$. Let $H_0=N(T),\; H_{+}=N(I-V)$ and $H_{-}=N(I+V)$. Then $H=H_0\oplus H_{+}\oplus H_{-}$. All these subspaces are invariant under $T$. Let $T_0=T|_{N(T)},\; T_{+}=T|_{H_{+}}$ and $T_ {-}=T|_{H_{-}}$. Then $T=T_0\oplus T_{+}\oplus T_{-}$. Further more, $T+$ is strictly positive, $T_{-}$ is strictly negative and $T_0=0$ if $N(T)\neq {\{0}\}$.  Let $P_0=P_{N(T)},\; P_{\pm}=P_{H_{\pm}}$. Then $P_0=I-V^2$ and $P_{\pm}=\frac{1}{2}(V^2\pm V)$. Thus $V=P_{+}-P_{-}$. For details see \cite[Example 7.1, page 139]{schmudgen}.  Note that the operators $T_{+}$ and $T_{-}$ are different than those used in Theorem \ref{sharpcharacterization}.
\begin{theorem}\label{selfadjstructure}
Let $T\in \mathcal {AN}(H)$ be self-adjoint with the polar decomposition $T=V|T|$. Then
\begin{enumerate}
\item \label{selfadjchar} the operator $T$ has the representation:
\begin{equation*}
T=K-F+\alpha V,
\end{equation*}
where $K\in \mathcal K(H),\; F\in \mathcal F(H)$ are self-adjoint  with $KF=0$ and $F^2\leq \alpha^2I$
\item \label{partialisometryANprop}if $T$ is not a compact operator, then $V\in \mathcal{AN}(H)$
\item \label{cptisometryreln} $K^2+2\alpha \text{Re}(VK)\geq 0$.
\end{enumerate}
\end{theorem}
\begin{proof}

Proof of (\ref{selfadjchar}): We prove this in two cases;\\
Case $1:$ $T$ one-to-one: In this case $H=H_{+}\oplus H_{-}$ and $T=T_{+}\oplus T_{-}$. Since $H_{\pm}$ reduces $T$, we have $T_{\pm}\in \mathcal B(H_{\pm})$. As $T\in \mathcal {AN}(H)$, we have that $T_{\pm}\in \mathcal{AN}(H_{\pm})$.
Hence By Theorem \ref{sharpcharacterization}, we have that $T_{+}=K_{+}-F_{+}+\alpha I_{H_{+}}$ such that $K_{+}$ is positive compact operator, $F_{+}$ is finite rank positive operator with the property that $K_{+}F_{+}=0$ and $F_{+}\leq \alpha I_{H_{+}} $. As $T_{+}$ is strictly positive, $\alpha>0$.

Similarly, $T_{-}\in \mathcal {AN}(H_{-})$ and strictly negative. Hence there exists a triple $(K_{-},F_{-},\beta)$ such that
$-T_{-}=K_{-}-F_{-}+\beta I_{H_{-}}$, where $K_{-}\in \mathcal K(H_{-})$ is positive, $F_{-}\in \mathcal F(H_{-})$ is positive with $K_{-}F_{-}=0,\; F_{-}\leq \beta I_{H_{-}}$ and  $\beta>0$.  Hence we can write $T_{-}=-K_{-}+F_{-}-\beta I_{H_{-}}$ and
\begin{align*}
T=\left(
    \begin{array}{cc}
      T_{+} & 0 \\
      0 & T_{-} \\
    \end{array}
  \right)&=\left(
            \begin{array}{cc}
              K_{+}-F_{+}+\alpha I_{H_{+}} & 0 \\
              0 & -K_{-}+F_{-}-\beta I_{H_{-}} \\
            \end{array}
          \right)\\
                  &=\left(
              \begin{array}{cc}
               K_{+} & 0 \\
                0 &  -K_{-} \\
              \end{array}
            \right)-\left(
                      \begin{array}{cc}
                        -F_{+} &0 \\
                        0& F_{-} \\
                      \end{array}
                    \right)+\left(
                              \begin{array}{cc}
                                \alpha I_{H_{+}} & 0 \\
                                0& \beta I_{H_{-}} \\
                              \end{array}
                            \right).
\end{align*}

We also have that
\begin{align*}
|T|&=\left(
    \begin{array}{cc}
      T_{+} & 0 \\
      0 & -T_{-} \\
    \end{array}
  \right)\\
  &=\left(
              \begin{array}{cc}
               K_{+} & 0 \\
                0 &  K_{-} \\
              \end{array}
            \right)-\left(
                      \begin{array}{cc}
                        F_{+} &0 \\
                        0& F_{-} \\
                      \end{array}
                    \right)+\left(
                              \begin{array}{cc}
                                \alpha I_{H_{+}} & 0 \\
                                0& \beta I_{H_{-}} \\
                              \end{array}
                            \right).
                            \end{align*}

                            Let  $K_1:=\left(
              \begin{array}{cc}
               K_{+} & 0 \\
                0 &  K_{-} \\
              \end{array}
            \right)$ and $F_1:=\left(
                      \begin{array}{cc}
                        F_{+} &0 \\
                        0& F_{-} \\
                      \end{array}
                    \right)$.
Then
\begin{equation*}
|T|=K_1-F_1-\left(
                              \begin{array}{cc}
                                \alpha I_{H_{+}} & 0 \\
                                0& \beta I_{H_{-}} \\
                              \end{array}
                            \right).
                            \end{equation*}
                            Clearly, $K_1F_1=0$, both $K_1$ and $F_1$ are positive, $F\leq \max{\{\alpha, \beta}\}I$.
                             By the uniqueness of the decomposition (see Theorem \ref{sharpcharacterization}), if $|T|=K_2-F_2+\gamma I$, then   we can conclude that  $K_1=K_2,\; F_1=F_2$ and $\alpha=\beta=\gamma$.   With this observation,  we have that

$\left(
   \begin{array}{cc}
     \alpha I_{H_{+}} & 0 \\
     0 & \beta I_{H_{-}} \\
   \end{array}
 \right)=\alpha(P_{+}-P_{-})=\alpha V$.

 Now taking $K:=\left(
              \begin{array}{cc}
               K_{+} & 0 \\
                0 &  -K_{-} \\
              \end{array}
            \right)$, $F:=\left(
                      \begin{array}{cc}
                        -F_{+} &0 \\
                        0& F_{-} \\
                      \end{array}
                    \right)$, we can write $T=K-F+\alpha V$.  Here $K$ is self-adjoint compact operator, $F$ is a self-adjoint finite rank operator with $KF=0$. Finally, it is easy to verify that $F^2\leq \alpha^2I$.

                    Next, we show that $V\in \mathcal{AN}(H)$. Since $T^{-1}$ exists, $T$ cannot be compact.  It suffices to prove $V^2\in \mathcal{AN}(H)$. We have $V^2=P_{+}+P_{-}=P_{R(T)}=I\in \mathcal{AN}(H)$.

 \noindent Case 2; $T$ need not be one-to-one:  In this case $T_0=0$ and $T=T_0\oplus T_{+}\oplus T_{-}$. Since  all the operators $T_0,T_{+}$ and $T_{-}$ are $\mathcal{AN}$-operators, we have that
 \begin{align*}
 T&=\left(
                               \begin{array}{ccc}
                                 T_{+} & 0 & 0 \\
                                 0 & T_{-} & 0 \\
                                 0 & 0 & T_0 \\
                               \end{array}
                             \right)\\
                             &=\left(
                               \begin{array}{ccc}
                                  K_{+}-F_{+}+\alpha I_{H_{+}}& 0 & 0 \\
                                 0 & -K_{-}+F_{-}-\beta I_{H_{-}} & 0 \\
                                 0 & 0 & 0 \\
                               \end{array}
                             \right)\\
                             &=\left(
                                 \begin{array}{ccc}
                                   K_{+} &0 &0 \\
                                  0 & K_{-} & 0 \\
                                   0 & 0 & 0 \\
                                 \end{array}
                               \right)-\left(
                                 \begin{array}{ccc}
                                   -F_{+} & 0 & 0 \\
                                   0 &F_{-} & 0 \\
                                   0 & 0 & 0 \\
                                 \end{array}
                               \right)+\left(
                                 \begin{array}{ccc}
                                   \alpha I_{H_{-}} &0 & 0 \\
                                   0 & -\alpha I_{H_{-}} &0 \\
                                   0 &0 &0 \\
                                 \end{array}
                               \right).
                              \end{align*}
(Following the same arguments as in Case $(1)$, we can show that $\alpha=\beta$)

Let $K=\left(
                                 \begin{array}{ccc}
                                   K_{+} &0 &0 \\
                                  0 & K_{-} & 0 \\
                                   0 & 0 & 0 \\
                                 \end{array}
                               \right)$ and $F=\left(
                                 \begin{array}{ccc}
                                   -F_{+} & 0 & 0 \\
                                   0 &F_{-} & 0 \\
                                   0 & 0 & 0 \\
                                 \end{array}
                               \right)$. Clearly, $V=\left(
                                 \begin{array}{ccc}
                                   I_{H_{+}} &0 & 0 \\
                                   0 & - I_{H_{-}} &0 \\
                                   0 &0 &0 \\
                                 \end{array}
                               \right)$.
Then $T=K-F+\alpha V$ and $K$ and $F$ satisfy the stated properties.

Proof of (\ref{partialisometryANprop}): Note that if $\alpha=0$, then $T$ is compact. If $\alpha>0$ and $V$ is a finite rank operator, then also $T$ can be compact. Hence assume that $\alpha>0$ and $R(V)$ is infinite dimensional. But by Theorem (\ref{propertiesofAN}), $N(T)=N(V)$ is finite dimensional.  So the conclusion follows by \cite[Proposition 3.14]{carvajalneves}.

Proof of (\ref{cptisometryreln}): As $VK=KV$, $KV$ is  self-adjoint. Hence  $K^2+2Re(V^*K)=K^2+2VK$.  Thus
\begin{equation*}
K^2+2VK=\left(
                                 \begin{array}{ccc}
                                   K_{+}^2+2K_{+} &0 &0 \\
                                  0 & K_{-}^2-2K_{-} & 0 \\
                                   0 & 0 & 0 \\
                                 \end{array}
                               \right).
\end{equation*}
Since the $(1,1)$ entry of the above matrix is positive, to get the conclusion, it suffices to prove that the $(2,2)$ entry is positive. Clearly, $K_{-}^2-2K_{-}$ is self-adjoint. Next, we show that $\sigma(K_{-}^2-2K_{-})$ is positive. Let $\lambda \in \sigma(K_{-})$. Then $\lambda \leq 0$ and $\lambda^2-2\lambda \in \sigma(K_{-}^2-2K_{-})$. But $\lambda^2-2\lambda=\lambda(\lambda-2)\geq 0$. Hence $K_{-}^2-2K_{-}$ is positive.
 \end{proof}

 \begin{corollary}
  Let $T=T^*\in \mathcal{AN}(H)$. Then $\sigma(T)$ is countable.
  \end{corollary}
  \begin{proof}
  Since $T=T_{+}\oplus T_{-}\oplus T_0$ and all these operators $T_{+},T_{-}$ and $T_0$ are $\mathcal{AN}$ operators. We know that $\sigma(T_{+}),\sigma(T_0)$ are countable, as they are positive. Also, $-T_{-}$ is positive $\mathcal {AN}$-operator and hence $\sigma(T_{-})$ is countable. Hence we can conclude that $\sigma(T)=\sigma(T_{+})\cup \sigma(T_{-})\cup \sigma(T_0)$ is countable.
  \end{proof}

Next, we can get the structure of normal $\mathcal{AN}$-operators. Here we use a different approach to  the one used in Theorem \ref{selfadjstructure}.

 \begin{proposition}\label{normalstructure}
 Let $T\in \mathcal{AN}(H)$ be normal with the polar decomposition $T=V|T|$. Then there exists a  compact normal  operator $K$, a  finite rank normal  operator $F\in \mathcal B(H)$ such that
 \begin{enumerate}
 \item $T$ has the representation: \label{normalstructurerepn}  \begin{equation}
  T=K-F+\alpha V
 \end{equation}
with   $KF=0$ and $F^*F\leq \alpha^2I$
\item\label{normalnecessary1} $K^*K+2\alpha \text{Re}(V^*K)\geq 0$
\item\label{normalnecessary2} $V,K,F$ commutes mutually
\item \label{normalnecessary3} if $\alpha>0$, then $V\in \mathcal {AN}(H)$.
\end{enumerate}
 \end{proposition}
\begin{proof}
Proof of (\ref{normalstructurerepn}):
 It is known that  $T$ is normal  if and only if  $V|T|=|T|V$.  Since $|T|\in \mathcal{AN}(H)$, we have $|T|=K_1-F_1+\alpha I$, where $K_1\in \mathcal K(H)$ is positive, $F_1\in \mathcal F(H)$ is positive and $F_1\leq \alpha I$.

 First, we show that $V$ is normal. We have $N(T^*)=N(T)=N(V)$. Hence
 \begin{equation*}
  V^*V=P_{N(V)^{\bot}}=P_{N(T)^{\bot}}=P_{N(T^*)^{\bot}}=P_{\overline{R(T)}}=P_{R(V)}=VV^*.
 \end{equation*}

So, $T=K-F+\alpha V$, where $K=VK_1$ and $F=VF_1$.  Next, we show that $K$ and $F$ are normal.  As $T$ is  normal,  $V$ commutes with $|T|$. Hence
\begin{equation}\label{commutingdifference}
V(K_1-F_1)=(K_1-F_1)V.
\end{equation}
Since $V$ commute with $K_1-F_1$, it also commute with $(K_1-F_1)^2$.  But, $(K_1-F_1)^2=K_1^2+F_1^2=(K_1+F_1)^2$. With this, we can conclude that $V(K_1+F_1)^2=(K_1+F_1)^2V$.  Hence,
\begin{equation}\label{commutingsum}
V(K_1+F_1)=(K_1+F_1)V.
\end{equation}
Thus by Equations (\ref{commutingdifference}) and (\ref{commutingsum}), we can conclude that $VK_1=K_1V$ and $VF_1=F_1V$.  By the Fuglede's theorem we can conclude that $V^*K_1=K_1V^*$ and $V^*F_1=F_1V^*$.
 Next,
\begin{equation*}
K^*K=K_1V^*VK_1=K_1VV^*K_1 =VK_1V^*K_1=VK_1K_1V^*=KK^*.
\end{equation*}
With  similar arguments we can show that $F$ is normal.

Next, we show that $KF=0$. Since $V$ commute with $K_1$ and $F_1$, we have $KF=VK_1VF_1=V^2K_1F_1=0$.

Finally, $F^*F=F_1V^*VF_1\leq \|V\|^2F_1^2\leq \alpha^2I$.

Proof of (\ref{normalnecessary1}): Using the relations $VK_1=K_1V$ and $V^*K_1=K_1V^*$, we get
\begin{align*}
K^*K+\alpha(V^*K+K^*V)&=K_1V^*VK_1+\alpha(V^*VK_1+K_1V^*V)\\
                                          &=V^*V(K_1^2+2\alpha K_1) \\
                                          &=P_{N(V)^{\bot}} (K_1^2+2\alpha K_1)\\
                                          &=K_1^2+2\alpha K_1\\
                                          &\geq 0.
\end{align*}
In the fourth step of the above equations we have used the fact that $P_{N(V)^{\bot}}K_1=P_{R(V)}K_1=P_{R(|T|)}K_1=K_1$.

Proof of (\ref{normalnecessary2}): We have $VK=VVK_1=VK_1V=KV$ and $VF=VVF_1=VF_1V=FV$. Also, $KF=0=FK$.

Proof of (\ref{normalnecessary3}): Note that by applying  (\ref{fdmlnullspace}) of  Proposition \ref{propertiesofAN}  to $|T|$, we can conclude that  $N(|T|)=N(T)=N(V)$ is finite dimensional. Now the conclusion follows by \cite[Proposition 3.14]{carvajalneves}.
\end{proof}
\begin{corollary}\label{ajointpreserving}
Let $T\in \mathcal B(H)$ be normal. Then $T\in \mathcal{AN}(H)$ if and only if $T^*\in \mathcal{AN}(H)$.
\end{corollary}
\begin{proof}
 We know that $T\in \mathcal{AN}(H)$ if and only if $T^*T\in \mathcal{AN}(H)$ by Corollary \ref{gramchar}. Since $T^*T=TT^*$, by Corollary \ref{gramchar} again, it follows that $TT^*\in \mathcal{AN}(H)$ if and only if $T^*\in \mathcal{AN}(H)$.
\end{proof}
\section{General case}
In this section we prove the structure of  absolutely norm attaining operators defined between two different Hilbert spaces.
\begin{theorem}\label{gencharacterization}
Let $T\in \mathcal {AN}(H_1,H_2)$ with the polar decomposition $T=V|T|$. Then
\begin{equation*}
T=K-F+\alpha V,
\end{equation*}
where $K\in \mathcal K(H_1,H_2), \; F\in \mathcal F(H_1,H_2)$ such that $K^*F=0=KF^*$ and $\alpha^2I\geq F^*F $.
\end{theorem}
\begin{proof}
Since $|T|\in \mathcal{AN}(H_1)$  and positive, we have by Theorem \ref{sharpcharacterization}, $|T|=K_1-F_1+\alpha I$, where  the triple $(K_1,F_1,\alpha)$ satisfy conditions in (\ref{orthogonalrepn}) of Theorem \ref{sharpcharacterization}.
Now, $T=K-F+\alpha V$, where $K=VK_1,\; F=VF_1$. Clearly,
\begin{align*}
K^*F=K_1V^*VF_1=K_1P_{N(V)^{\bot}}F_1&=K_1(I-P_{N(V)})F_1\\
                                     &=K_1F_1-K_1P_{N(V)}F_1\\ &=0   \;(\text{since }\; N(V)=N(|T|)\subseteq N(K_1))  .
\end{align*}
Also, clearly, $KF^*=VK_1F_1V^*=0$.

Finally, $F^*F=F_1V^*VF_1\leq  \|V^*V\|F_1^2\leq F_1^2 \leq \alpha^2 I$.
\end{proof}
\begin{proposition}
Let $T\in \mathcal B(H)$ and $U\in \mathcal B(H)$ be unitary such that $T^*=U^*TU$. Then $T\in \mathcal {AN}(H)$ if and only if $T^*\in \mathcal {AN}(H)$.
\end{proposition}
\begin{proof}
This follows by \cite[Theorem 3.5]{carvajalneves}.
\end{proof}

Next, we discuss a possible  converse in the general case.

\begin{theorem}\label{generalconverse}
Let $K\in \mathcal K(H_1,H_2),\; F\in \mathcal F(H_1,H_2)$, $\alpha \geq 0$ and $V\in \mathcal B(H_1,H_2)$ be a partial isometry. Further assume that
\begin{enumerate}
 \item $V\in \mathcal{AN}(H_1,H_2)$
 \item $ K^*K+\alpha (V^*K+K^*V)\geq 0$.
\end{enumerate}
Then
$T:=K-F+\alpha V \in \mathcal {AN}(H_1,H_2)$.
 \end{theorem}

 \begin{proof}
  If $\alpha=0$, then $T\in \mathcal{K}(H_1,H_2)$. Hence $T\in \mathcal {AN}(H_1,H_2)$. Next assume that $\alpha>0$. We prove this case by showing $T^*T\in \mathcal{AN}(H_1)$. By a simple calculation we can get
$   T^*T=\mathcal K-\mathcal F +\alpha^2 P_{N(V)^{\bot}}$,   where,
   \begin{equation*}
   \mathcal K= K^*K+\alpha (V^*K+K^*V),\quad  \mathcal F= F^*F-F^*K-K^*F-\alpha (V^*F+F^*V).
  \end{equation*}
Since  $V\in \mathcal{AN}(H_1,H_2)$, either $N(V)$ or $N(V)^{\bot}$ is finite dimensional. If $N(V)^{\bot}$ is finite dimensional, then $T^*T\in \mathcal K(H_1)$. Hence $T\in \mathcal K(H_1,H_2)$.

If $N(V)$ is finite dimensional, then $T^*T=\mathcal K-(\mathcal F-\alpha^2 P_{N(V)})+\alpha^2I$. Note that the operator $\mathcal F-\alpha^2 P_{N(V)}$ is a finite rank self-adjoint operator.
Hence $T^*T\in \mathcal{AN}(H_1)$ by Theorem \ref{structurethm2nd}. Now the conclusion follows by Corollary \ref{gramchar}.
 \end{proof}

\begin{corollary}
  Suppose that $K\in \mathcal K(K)$, $ F\in \mathcal F(H)$  are  normal and $V\in \mathcal B(H)$ is a normal partial isometry such that $V,F, K$ commute mutually. Let $\alpha\geq 0$. Then
\begin{enumerate}
\item\label{normalityofsum} $T:=K-F+\alpha V$ is normal and

\item\label{normalANconverse} if  $K^*K+2\alpha V^*K \geq 0$ and $V\in \mathcal {AN}(H)$, then $T\in \mathcal {AN}(H)$.
\end{enumerate}
\end{corollary}
\begin{proof}
To prove (\ref{normalityofsum}) we observe that if $A$ and $B$ are commuting normal operators, then $A+B$ is normal (see \cite[Page 342, Exercise 12]{rudin} for details). By this observation it follows that $T$ is normal.

To prove (\ref{normalANconverse}), since $VK=KV$, by Fuglede's theorem \cite[Page 315]{rudin}, $V^*K=KV^*$. With this observation and Theorem \ref{generalconverse}, the conclusion follows.
\end{proof}

\begin{corollary}
Suppose that $K\in \mathcal K(H)$, $ F\in \mathcal F(H)$  are  self-adjoint  and $V\in \mathcal B(H)$ is a self-adjoint,  partial isometry such that
\begin{itemize}
\item[(a)] $V\in \mathcal{AN}(H)$
\item[(b)] $K^2+2\alpha (VK)\geq 0$.
\end{itemize}
Then $T:=K-F+\alpha V$ is self-adjoint and $\mathcal{AN}$-operator.
\end{corollary}
\begin{proof}
The proof directly follows by Theorem \ref{generalconverse}.
\end{proof}
 \begin{definition}\cite[page 349]{jbconway}\label{leftsemifredholm}
Let $T\in \mathcal B(H_1,H_2)$. Then $T$ is called left \textit{semi-Fredholm} if there exists a $B\in \mathcal B(H_2,H_1)$ and $K\in \mathcal K(H_1)$ such that $BT=K+I$ and \textit{right semi-Fredholm} if there exists a $A\in \mathcal B(H_2,H_1)$ and $K'\in \mathcal K(H_2)$ such that $TA=K'+I$.

If $T$ is both left semi-Fredholm and right semi-Fredholm, then $T$ is called Fredholm.
\end{definition}
\begin{remark}\label{fredholmadjoint}
Note that $T$ is left semi-Fredholm if and only if $T^*$ is right semi-Fredholm (see \cite[section 2, page 349]{jbconway} for details).
\end{remark}

\begin{corollary}\label{semifredholm}
Let $T\in \mathcal{AN}(H_1,H_2)$ but not compact. Then $T$ is left-semi-Fredholm.
\end{corollary}
\begin{proof}
Let $T=V|T|$ be the polar decomposition of $T$. Then $|T|=V^*T$. As,  $|T|\in \mathcal{AN}(H_1)$, by Theorem \ref{sharpcharacterization}, there exists a  triple $(K,F,\alpha)$ satisfying conditions in Theorem \ref{sharpcharacterization}, such that
$V^*T=K-F+\alpha I$. Let $K^{'}=K-F$. Then $V^*T=K^{'}+\alpha I$. By Definition \ref{leftsemifredholm}, it follows that $T$ is left-semi-Fredholm.
\end{proof}

\bibliographystyle{plain}

\end{document}